\documentclass[12pt]{amsart}

\usepackage{amsmath,amsthm,amssymb}
\usepackage{graphicx}

\usepackage{a4wide}
\usepackage{enumerate}

\newtheorem{theorem}{Theorem}[section]

\newtheorem{proposition}[theorem]{Proposition}
\newtheorem{corollary}[theorem]{Corollary}

\theoremstyle{remark}

\newtheorem*{claim*}{Claim}

\newcommand{\FM}{\ensuremath{\mathcal{F}}}
\newcommand{\CM}{\ensuremath{\mathcal{C}}}
\newcommand{\TM}{\ensuremath{\mathcal{T}}}

\newcommand{\SM}{\ensuremath{\mathcal{S}}}

\newcommand{\R}{\ensuremath{\mathbb{R}}}
\newcommand{\W}{\ensuremath{\mathbb{W}}}
\newcommand{\M}{\ensuremath{\mathbb{L}}}
\newcommand{\dS}{\ensuremath{\mathbb{S}}}
\newcommand{\g}[1]{\ensuremath{\mathfrak{#1}}}

\DeclareMathOperator{\Ad}{Ad}

\DeclareMathOperator{\Exp}{Exp}

\begin{document}
\title{Cohomogeneity one actions on Minkowski spaces}

\author[J.\ Berndt]{J\"{u}rgen Berndt}
\author[J.\ C.\ D\'{\i}az-Ramos]{Jos\'{e} Carlos D\'{\i}az-Ramos}
\author[M.\ Vanaei]{Mohammad Javad Vanaei}

\address{Department of Mathematics, King's College London, United Kingdom}
\email{j.berndt@kcl.ac.uk}

\address{Department of Geometry and Topology, University of Santiago de Compostela, Spain}
\email{josecarlos.diaz@usc.es}

\address{Department of Pure Mathematics, Tarbiat Modares University, Iran}
\email{javad.vanaei@modares.ac.ir}

\thanks{The second author has been supported 
by projects MTM2009-07756 and INCITE 09 207151PR (Spain).}

\begin{abstract}
We study isometric cohomogeneity one actions on the $(n+1)$-dimensional Minkowski space $\M^{n+1}$ up to orbit-equivalence. We give examples of isometric cohomogeneity one actions on $\M^{n+1}$ whose orbit spaces are non-Hausdorff. We show that there exist isometric cohomogeneity one actions on $\M^{n+1}$, $n \geq 3$, which are orbit-equivalent on the complement of an $n$-dimensional degenerate subspace $\W^n$ of $\M^{n+1}$ and not orbit-equivalent on $\W^n$. We classify isometric cohomogeneity one actions on $\M^2$ and $\M^3$ up to orbit-equivalence.
\end{abstract}

\date{\today}

\subjclass[2010]{53C50}

\keywords{Cohomogeneity one actions, Minkowski space, parabolic subgroups}

\maketitle

\section{Introduction}
\label{sect:Intro}

Cohomogeneity one actions are known to be useful for the construction of geometric structures on manifolds, e.g.\ metrics with special holonomies, or for calculating explicit solutions of certain systems of partial differential equations, e.g.\ the Einstein equations. In Riemannian geometry, the orbit structure of a cohomogeneity one action is easy to describe.
Let $M$ be a connected complete Riemannian manifold and let $H$ be a connected subgroup of the isometry group of $M$. Assume that the action is proper and of cohomogeneity one, that is, the codimension of a principal orbit of the action is one. It is well-known (see e.g.\ \cite{BB82} or \cite{M57}) that the orbit space of such an action is homeomorphic to the real line ${\mathbb R}$, to the circle $S^1$, to the closed unbounded interval $[0,\infty)$, or to the closed bounded interval $[0,1]$. If the orbit space is homeomorphic to ${\mathbb R}$ or to $S^1$, then the orbits form a Riemannian foliation. If the orbit space is homeomorphic to $[0,\infty)$, then there exists exactly one singular orbit and the principal orbits are the tubes around this singular orbit. If the orbit space is homeomorphic to $[0,1]$, then there exist exactly two singular orbits and each principal orbit is a tube around each of the two singular orbits. A particular consequence of this is that all orbits of a cohomogeneity one action can be constructed from one orbit of the action, and it does not matter whether the orbit is principal or singular. The fact that the orbit space is one-dimensional can sometimes be used for reformulating systems of partial differential equations in terms of ordinary differential equations. The classification of cohomogeneity one actions on certain manifolds has also attracted much attention. For cohomogeneity one actions on Riemannian symmetric spaces see for example \cite{BT13} (for the noncompact case) and \cite{Ko02} (for the compact case).

The motivation for this paper is to get a better understanding of cohomogeneity one actions in Lorentzian geometry. We mention that, in contrast to cohomogeneity one actions, transitive isometric actions in Lorentzian geometry have been studied quite thoroughly, see for example the papers \cite{AS97} and \cite{AS01} by Adams and Stuck. In this paper we investigate cohomogeneity one actions on the $(n+1)$-dimensional Minkowski space $\M^{n+1}$. Ahmadi and Kashani investigated in \cite{AK11} such actions under the assumption that the action is proper. This situation is similar to the Riemannian case and  the orbit space is homeomorphic to $\R$ or to $[0,\infty)$. We will not assume here that the action is proper.

One interesting class of cohomogeneity one actions on $\M^{n+1}$ is given by certain subgroups of a maximal parabolic subgroup $Q$ of the restricted Lorentz group $SO^o_{n,1}$. The restricted Lorentz group acts transitively on the real hyperbolic space $H^n$, considered as a space-like hypersurface in $\M^{n+1}$ in the usual way, so that we can write $H^n = SO^o_{n,1}/SO_n$ as a homogeneous space. Consider an Iwasawa decomposition $SO^o_{n,1} = SO_n A N$. Then the solvable Lie group $AN$ acts transitively on $H^n$. The maximal parabolic subgroup $Q$ is, up to conjugacy,  of the form $Q = K_0AN$ with $K_0 \cong SO_{n-1} \subset SO_n$. The parabolic subgroup $Q = K_0AN$ acts with cohomogeneity one on $\M^{n+1}$. Our first main result, for $n \geq 3$, is Theorem \ref{nparabolic}. This results states that every subgroup $H = K'AN \subset K_0AN$ acts on $\M^{n+1}$ with cohomogeneity one. We investigate thoroughly the orbit structure of these actions. A remarkable feature of these actions is that there exists an $n$-dimensional degenerate subspace $\W^n$ of $\M^{n+1}$ such that on $\M^{n+1} \setminus \W^n$ all these actions have the same orbits, whereas the orbit structures become different on the $n$-dimensional subspace $\W^n$. As a consequence we see that even if the orbit structure of a cohomogeneity one action on $\M^{n+1}$ is known on a dense and open subset, one cannot reconstruct in general all orbits. Such a curious phenomenon cannot occur in Riemannian geometry.

Our second main result is Theorem \ref{th:L3}, which contains an explicit classification of all cohomogeneity one actions on $\M^3$ up to orbit-equivalence. We show that, up to orbit-equivalence, there is a one-parameter family of such actions, parametrized by $[0,\infty)$, plus nine further cohomogeneity one actions. We investigate the orbit structures and the geometry of the orbits of these actions in detail. It is worthwhile to compare Theorem \ref{th:L3} with its Euclidean counterpart (\cite{So18}): There are, up to orbit-equivalence, exactly three cohomogeneity one actions on the $3$-dimensional Euclidean space ${\mathbb E}^3$. The orbits are either parallel planes, concentric spheres or coaxial circular cylinders.

The paper is organized as follows. In Section \ref{sect:preliminaries} we present some basic material about the Minkowski space $\M^{n+1}$. In Section \ref{sect:isotropy action} we describe the orbit structure of the action of the restricted Lorentz group $SO^o_{n,1}$ on $\M^{n+1}$, or equivalently, of the isotropy representation of the homogeneous space $\M^{n+1} = (SO^o_{n,1} \ltimes \M^{n+1}) / SO^o_{n,1}$. In Section \ref{sect:parabolic action} we investigate the action of a maximal parabolic subgroup $Q$ of $SO^o_{n,1}$, and some of its subgroups, on $\M^{n+1}$. This leads to our first main result Theorem \ref{nparabolic} and the curious phenomenon described above. In Section \ref{sect:L2} we determine all cohomogeneity one actions on the Minkowski plane $\M^2$ up to orbit-equivalence. Finally, in Section \ref{sect:L3}, we determine all cohomogeneity one actions in the Minkowski space $\M^3$ up to orbit-equivalence.

We would like to thank Miguel S\'{a}nchez Caja for helpful discussions and suggestions.

\section{Preliminaries}\label{sect:preliminaries}

We denote by $\M^{n+1}$ the $(n+1)$-dimensional Minkowski space ($n \geq 1$) with the usual orientation, coordinates and inner product
\[
\langle u,v \rangle = \sum_{i=1}^n u_iv_i - u_{n+1}v_{n+1}.
\]
We denote by $e_1,\ldots,e_n,e_{n+1}$ the standard orthonormal basis of $\M^{n+1}$.

The orthogonal group $O_{n,1}$ of the above inner product is also known as the Lorentz group of $\M^{n+1}$ and the elements of it are so-called Lorentz transformations of $\M^{n+1}$. The isometry group $I(\M^{n+1})$ of $\M^{n+1}$ is the semidirect product
$I(\M^{n+1}) = O_{n,1} \ltimes_\tau \M^{n+1}$
with $\tau : O_{n,1} \times \M^{n+1} \to \M^{n+1}\ ,\ (x,u) \mapsto xu = x(u)$. The multiplication and inversion on $I(\M^{n+1})$ is given by
$(x,u)(y,v) = (xy,u+xv)$ and $(x,u)^{-1} = (x^{-1},-x^{-1}u)$
and the action of $I(\M^{n+1})$ on $\M^{n+1}$ is given by
$I(\M^{n+1}) \times \M^{n+1} \to \M^{n+1}\ ,\ ((x,u),p) \mapsto xp+u$.
The isometry group $I(\M^{n+1})$ has four connected components, corresponding to preserving and reversing space- and time-orientation respectively. We denote by $I^o(\M^{n+1}) = SO^o_{n,1} \ltimes \M^{n+1}$ the identity component of $I(\M^{n+1})$, where $SO^o_{n,1}$ is the subgroup of $O_{n,1}$ preserving both space- and time-orientation of $\M^{n+1}$. The connected noncompact real Lie group $SO^o_{n,1}$ is also known as the restricted Lorentz group of $\M^{n+1}$. For $n=1$ this is a one-dimensional abelian Lie group and for $n \geq 2$ it is a simple Lie group. The restricted Lorentz group $SO^o_{n,1}$ is a normal subgroup of the Lorentz group $O_{n,1}$.

The Lie algebra of $I(\M^{n+1})$ is the semidirect sum $\g{so}_{n,1} \oplus_\phi \M^{n+1}$ with
$\phi : \g{so}_{n,1} \times \M^{n+1} \to \M^{n+1}\ ,\ (X,u) \mapsto Xu = X(u)$. The Lie bracket on $\g{so}_{n,1} \oplus_\phi \M^{n+1}$ is given by
\[
[X+u,Y+v] = [X,Y]+(Xv-Yu) = (XY-YX)+(Xv-Yu).
\]
From this we get the adjoint representation
\[
\Ad((x,u))(Y+v) = xYx^{-1} + (xv - (xYx^{-1})u).
\]
The Lie algebra $\g{so}_{n,1}$ of the Lorentz group is given by
\[
\g{so}_{n,1} = \left\{X = \begin{pmatrix} B   &   b\\ b^t   &   0 \end{pmatrix} : B\in\g{so}_n,\ b\in\R^n \right\}.
\]
The Cartan involution $\theta(X) = -X^t$ of $\g{so}_{n,1}$ induces the Cartan decomposition
$\g{so}_{n,1}=\g{k}\oplus\g{p}$ with
\begin{align*}
\g{k} &= \left\{\begin{pmatrix} B  &  0 \\ 0  &  0 \end{pmatrix} : B\in\g{so}_n\right\} \cong \g{so}_n\ ,
&\g{p}&=\left\{\begin{pmatrix} 0   &   b\\ b^t   &   0 \end{pmatrix} : b\in\R^n \right\} \cong \R^n.
\end{align*}
The subspace
\[
\g{a} = \R \begin{pmatrix} 0   &   e_n \\ (e_n)^t   &   0 \end{pmatrix} \subset \g{p}
\]
is a maximal abelian subspace of $\g{p}$. Let $\g{so}_{n,1}=\g{g}_{-\alpha}\oplus\g{g}_0\oplus\g{g}_\alpha$
be the restricted root space decomposition of $\g{so}_{n,1}$ induced by $\g{a}$. Explicitly, we have
\begin{align*}
\g{g}_\alpha &=\left\{
\begin{pmatrix}
0   &   b   &   b \\
-b^t   &   0   &   0 \\
b^t   &   0  &   0
\end{pmatrix}
:b\in\R^{n-1}\right\}\cong\R^{n-1}\ ,\qquad \g{g}_{-\alpha}=\theta\g{g}_\alpha\ ,\\
\g{g}_0 &= \g{k}_0\oplus\g{a}\ ,\qquad \g{k}_0=\left\{
\begin{pmatrix}
B & 0 & 0 \\
0 & 0 & 0 \\
0 & 0 & 0
\end{pmatrix}:B\in\g{so}_{n-1}\right\}\cong\g{so}_{n-1}.
\end{align*}
Then $\g{n}=\g{g}_\alpha$ is an abelian subalgebra of $\g{so}_{n,1}$ and $\g{so}_{n,1}=\g{k}\oplus\g{a}\oplus\g{n}$
is an Iwasawa decomposition of $\g{so}_{n,1}$. The subalgebra
\[
\g{a} \oplus \g{n} =
\left\{
\begin{pmatrix}
0   &   b   &   b \\
-b^t   &   0   &   c \\
b^t   &   c  &   0
\end{pmatrix}:b\in\R^{n-1},\ c \in \R \right\}
\]
is a solvable subalgebra of $\g{so}_{n,1}$ and
\[
\g{k}_0 \oplus \g{a} \oplus \g{n} =
\left\{
\begin{pmatrix}
B   &   b   &   b \\
-b^t   &   0   &   c \\
b^t   &   c  &   0
\end{pmatrix}:B \in \g{so}_{n-1},\ b\in\R^{n-1},\ c \in \R \right\}
\]
is a parabolic subalgebra of $\g{so}_{n,1}$. We denote by $K \cong SO_n$, $K_0 \cong SO_{n-1}$, $A$ and $N$ the connected closed subgroups of $SO^o_{n,1}$ with Lie algebras $\g{k}$, $\g{k}_0$, $\g{a}$ and $\g{n}$ respectively. Then $K_0AN$ is a parabolic subgroup of $SO^o_{n,1}$ and $AN$ is a solvable subgroup of $SO^o_{n,1}$.

\section{The action of $SO^o_{n,1}$ on $\M^{n+1}$}\label{sect:isotropy action}

In this section we study the orbits of the action of the isotropy group $SO^o_{n,1}$ of $I^o(\M^{n+1})$ on $\M^{n+1}$. We give a detailed description of the orbit structure since it will be useful for investigating other isometric actions on $SO^o_{n,1}$. We first introduce some notations:
\begin{align*}
\SM^{n+1}  &= \{v \in \M^{n+1} : \langle v , v \rangle > 0\}
&\TM^{n+1}  &= \{v \in \M^{n+1} : \langle v , v \rangle < 0\},\\
\CM^n  &= \{v \in \M^{n+1} : \langle v , v \rangle = 0\},\\
\TM_+^{n+1}  &= \{v \in \TM^{n+1} : \langle v , e_{n+1} \rangle > 0\},
&\TM_-^{n+1}  &= \{v \in \TM^{n+1} : \langle v , e_{n+1} \rangle < 0\},\\
\CM_+^n  &= \{v \in \CM^n : \langle v , e_{n+1} \rangle > 0\},
&\CM_-^n  &= \{v \in \CM^n :  \langle v , e_{n+1} \rangle < 0\}.
\end{align*}
$\SM^{n+1}$ is the set of space-like vectors in $\M^{n+1}$, $\TM^{n+1}$ is the set of time-like vectors in $\M^{n+1}$ and $\CM^n$ is the set of light-like vectors in $\M^{n+1}$. The index $\pm$ refers to time-orientation. For $\M^2$ we also introduce
\begin{align*}
\SM^2_+  &= \{v \in \SM^2 : \langle v , e_1 \rangle > 0\},
&\SM^2_-  &= \{v \in \SM^2 :  \langle v , e_1 \rangle < 0\},\\
\CM^1_{++}  &= \{v \in \CM^1_+ : \langle v , e_1 \rangle > 0\},
&\CM^1_{+-}  &= \{v \in \CM^1_- : \langle v , e_1 \rangle > 0\},\\
\CM^1_{-+}  &= \{v \in \CM^1_+ : \langle v , e_1 \rangle < 0\},
&\CM^1_{--}  &= \{v \in \CM^1_- : \langle v , e_1 \rangle < 0\}.
\end{align*}

For $r \in \R_+$ we define
\begin{align*}
H^n_+(r)  &= \{v \in \TM^{n+1}_+ : \langle v , v \rangle = -r^2\},
&H^n_-(r)  &= \{v \in \TM^{n+1}_- : \langle v , v \rangle = -r^2\}.
\end{align*}
For $n \geq 2$, the induced metric on $H^n_+(r)$ is Riemannian and, for $n \geq 2$, $H^n_+(r)$ is the well-known hyperboloid model of $n$-dimensional real hyperbolic space with constant curvature $-r^{-2}$. We have $I(H^n_+(r)) = SO_{n,1}$. In particular, $H^n_+(r)$ is an orbit of $SO^o_{n,1}$. The isotropy group at a point is isomorphic to $K = SO_n$ and therefore, as a homogeneous space, $H^n_+(r) = SO^o_{n,1}/SO_n = SO^o_{n,1}/K$.
The set $H^n_-(r)$ is the image of $H^n_+(r)$ under the time-reversing isometry $\M^{n+1} \to \M^{n+1},(u_1,\ldots,u_n,u_{n+1}) \mapsto  (u_1,\ldots,u_n,-u_{n+1})$,
which implies that $H^n_-(r)$ is another orbit of $SO^o_{n,1}$ and therefore we again have $H^n_-(r) = SO^o_{n,1}/SO_n = SO^o_{n,1}/K$
as a homogeneous space.

For $r \in \R_+$ and $n \geq 2$ we define
\[
\dS^n(r) = \{v \in \SM^n : \langle v , v \rangle = r^2\},
\]
and for $n = 1$ we put
\begin{align*}
\dS^1_+(r)  &= \{v \in \SM^1_+ : \langle v , v \rangle = r^2\},
&\dS^1_-(r)  &= \{v \in \SM^1_- : \langle v , v \rangle = r^2\}.
\end{align*}
The induced metric on $\dS^n(r)$ is Lorentzian and $\dS^n(r)$ is the well-known hyperboloid model of $n$-dimensional de Sitter space with constant curvature $r^{-2}$. We have $I(\dS^n(r)) = O_{n,1}$. In particular, $\dS^n(r)$ is an orbit of $SO^o_{n,1}$ and the isotropy group at a point of $\dS^n(r)$ is isomorphic to $SO^o_{n-1,1}$. Thus, as a homogeneous space, we have $\dS^n(r) = SO^o_{n,1}/SO^o_{n-1,1}$. Topologically, $\dS^n(r)$ is homeomorphic to $\R \times S^{n-1}$, the product of a line and an $(n-1)$-dimensional sphere. For $n=1$ the sphere $S^0$ consists just of two points, which is the reason why we need to treat this special case separately. In this case both $\dS^1_+(r)$ and $\dS^1_-(r)$ are orbits of the one-dimensional Lie group $SO^o_{1,1}$ and the isotropy groups are trivial, that is, as homogeneous spaces we have $\dS^1_+(r) = SO^o_{1,1}$ and $\dS^1_-(r) = SO^o_{1,1}$.

Finally, $SO^o_{n,1}$ leaves $\CM^n$ invariant. For $n \geq 2$ the orbits of the action are the single point $\{0\}$ and the two light cones $\CM^n_+$ and $\CM^n_-$. The isotropy group of $SO^o_{n,1}$ at a point in $\CM^n_+$ or $\CM^n_-$ is isomorphic to the subgroup $K_0N$ of the parabolic subgroup $K_0AN$ of $SO^o_{n,1}$. Thus, as homogeneous spaces, we have $\CM^n_+ = SO^o_{n,1}/K_0N$ and $\CM^n_- = SO^o_{n,1}/K_0N$.
Note that $K_0N$ is isomorphic to the special Euclidean group $SO_{n-1} \ltimes \R^{n-1}$ of $\R^{n-1}$. For $n = 1$ the orbits of the action are the single point $\{0\}$ and the four light rays $\CM^1_{++}$, $\CM^1_{+-}$, $\CM^1_{-+}$ and $\CM^1_{--}$.

Altogether it follows that we have the following decomposition $\FM_{SO^o_{n,1}}$of $\M^{n+1}$ into orbits of $SO^o_{n,1}$. For $n \geq 2$ we get
\[
\FM_{SO^o_{n,1}} = \{0\} \cup \CM^n_\pm \cup \bigcup_{r \in \R_+} H^n_\pm(r) \cup \bigcup_{r \in \R_+} \dS^n(r),
\]
and for $n = 1$ we get
\[
\FM_{SO^o_{1,1}} = \{0\} \cup \CM^1_{\pm\pm} \cup \bigcup_{r \in \R_+} H^1_\pm(r) \cup \bigcup_{r \in \R_+} \dS^1_\pm(r).
\]
It is easy to see that for all $n \geq 1$ the orbit space with the quotient topology is not a Hausdorff space.

\section{The action of the parabolic subgroup $K_0AN$ of $SO^o_{n,1}$ on $\M^{n+1}$}\label{sect:parabolic action}

In this section we assume $n \geq 2$. The noncompact simple real Lie group $SO^o_{n,1}$ has, up to conjugacy, exactly one parabolic subgroup, namely $Q = K_0AN$. As subgroups of $SO^o_{n,1}$, we have
\begin{align*}
K_0 & =
\left\{
\begin{pmatrix}
B & 0 & 0 \\
0 & 0 & 0 \\
0 & 0 & 0
\end{pmatrix}:B \in SO_{n-1}\right\}\ , \\
A & =
\left\{
\begin{pmatrix}
I_{n-1} & 0 & 0 \\
0 & \cosh(t) & -\sinh(t) \\
0 & -\sinh(t) & \cosh(t)
\end{pmatrix}:t \in \R\right\}\ , \\
N & =
\left\{
\begin{pmatrix}
I_{n-1} & b & b \\
-b^t & 1-\frac{1}{2}b^tb & -\frac{1}{2}b^tb \\
b^t & \frac{1}{2}b^tb & 1 + \frac{1}{2}b^tb
\end{pmatrix}:b \in \R^{n-1}\right\}\ .
\end{align*}

The solvable subgroup $AN$ of $SO^o_{n,1}$ acts transitively on the hyperbolic spaces $H^n_+(r)$ and $H^n_-(r)$. This implies that $Q$, and every subgroup $K'AN \subset K_0AN$ of $Q = K_0AN$ with $K'\subset K_0$, acts transitively on the hyperbolic spaces $H^n_+(r)$ and $H^n_-(r)$.

The special Euclidean group $K_0N$ fixes the vector $w_0 = e_n - e_{n+1} \in \CM^n_-$ and the orbit $Q \cdot w_0$ is equal to $Q \cdot w_0 = A \cdot w_0 = \R_+w_0 = \R w_0 \cap \CM^n_-$.
Similarly, we have $Q \cdot (-w_0) = A \cdot (-w_0) = \R_-w_0 = \R w_0 \cap \CM^n_+$.
The solvable subgroup $AN$ acts transitively on $\CM^n_+ \setminus \R_-w_0$ and on $\CM^n_- \setminus \R_+w_0$. Altogether this implies that $Q$, and every subgroup $K'AN \subset K_0AN$ of $Q = K_0AN$ with $K'\subset K_0$, has exactly two orbits on the positive light cone $\CM^n_+$, namely $\R_-w_0$ and $\CM^n_+ \setminus \R_-w_0$. The argument is analogous for the negative light cone $\CM^n_-$.

The situation becomes more interesting when restricting the action to the de Sitter space $\dS^n(r)$. We define an $n$-dimensional degenerate subspace $\W^n$ of $\M^{n+1}$ by
\[
\W^n = \R^{n-1} \oplus \R w_0 =  \R e_1 \oplus \ldots \oplus \R e_{n-1} \oplus \R (e_n - e_{n+1}) .
\]
The intersection $\W^n \cap \CM^n_-$ is precisely the orbit $Q \cdot w_0 = AN \cdot w_0$. The intersection $\W^n \cap \dS^n(r)$ is equal to the cylinder $Z^{n-1}(r)$ defined by $\W^n \cap \dS^n(r) = S^{n-2}(r) \times \R w_0$,
where $S^{n-2}(r)$ is the $(n-2)$-dimensional Euclidean sphere with radius $r$ in $\R^{n-1} \subset \W^n$. For $n = 2$ the cylinder $Z^{n-1}(r)$ is the union of the two disjoint lines $re_1 + \R w_0$ and $-re_1 + \R w_0$. Each of the two lines is an orbit of $AN = Q$ (note that $K_0 = \{I_3\}$ if $n = 2$). The complement of these two lines in $\dS^2(r)$ consists of two connected components, and each of them is an orbit of $AN = Q$. We thus have:

\begin{proposition}\label{2parabolic}
The action of the parabolic subgroup $Q = AN$ of $SO^o_{2,1}$ on $\M^3$ is of cohomogeneity one. The orbits are
\begin{itemize}
\item[(i)] The hyperbolic planes $H^2_+(r)$ and $H^2_-(r)$, $r \in \R_+$;
\item[(ii)] The single point $\{0\}$, the two open rays $\R_+w_0$ and $\R_-w_0$ and their complements $\CM^2_+ \setminus \R_-w_0$ and $\CM^2_- \setminus \R_+w_0$, where $w_0 = e_2 - e_3$;
\item[(iii)] The two lines $re_1 + \R w_0$ and $-re_1 + \R w_0$ and the two connected components of the complement of these two lines in $\dS^2(r)$, $r \in \R_+$.
\end{itemize}
\end{proposition}

If $n > 2$, then $Z^{n-1}(r)$ is connected and the complement $\dS^n(r) \setminus Z^{n-1}(r)$ of the cylinder $Z^{n-1}(r)$ in the de Sitter space $\dS^n(r)$ has two connected components. The action of $K_0$, $A$ and $N$ on a point $x + sw_0 \in Z^{n-1}(r)$ is given by
\begin{align*}
x + sw_0  &\mapsto Bx + sw_0,
&x + sw_0  &\mapsto x + e^tsw_0,
&x + sw_0  &\mapsto x + (s+b^tx)w_0,
\end{align*}
with $x \in S^{n-2}(r) \subset \R^{n-1}$ and $s \in \R$, where $B \in SO_{n-1}$, $t \in \R$ and $b \in \R^{n-1}$, respectively. It follows that $K_0AN$ leaves the cylinder $Z^{n-1}(r)$ invariant. More precisely, $K_0 = SO_{n-1}$ acts canonically on $S^{n-2}(r)$ and trivially on $\R w_0$, $A$ and $N$ act trivially on $S^{n-2}(r)$, $N$ acts transitively on $\R w_0$, and $A$ has three orbits on $\R w_0$ (namely $\{0\}$, $\R_+ w_0$ and $\R_- w_0$). This shows that the orbits of $AN$ on $Z^{n-1}(r)$ are precisely the lines $p + \R w_0$ with $p \in S^{n-2}(r) \subset Z^{n-1}(r)$. Since $K_0$ acts transitively on $S^{n-2}(r)$ we see that the parabolic subgroup $Q = K_0AN$ acts transitively on $Z^{n-1}(r)$. If $K'$ is a subgroup of $K$, then the orbits of $K'AN$ on the cylinder $Z^{n-1}(r)$ correspond bijectively to the orbits of $K'$ on the sphere $S^{n-2}(r)$.

The orbit of $AN$ through $re_n$ consists of all points of the form
\[
r\left(b_1e_1 + \ldots + b_{n-1}e_{n-1} + \cosh(t)\left(1-\frac{1}{2}|b|^2\right)e_n - \frac{1}{2}\sinh(t)|b|^2e_{n+1}\right) \in \dS^n(r)
\]
with $t \in \R$ and $b \in \R^{n-1}$, which is exactly one of the two connected components of $\dS^n(r) \setminus Z^{n-1}(r)$. The other connected component can be obtained by taking the orbit of $AN$ through $-re_n$. Since $K_0AN$ leaves $Z^{n-1}(r)$ invariant we conclude that every subgroup $K'AN \subset K_0AN$ with $K' \subset K_0$ acts transitively on each of the two connected components of $\dS^n(r) \setminus Z^{n-1}(r)$. However, the action of such a subgroup $K'AN$ on the cylinder $Z^{n-1}(r)$ is not transitive in general. Altogether we can now conclude:

\begin{theorem}\label{nparabolic}
Let $K'$ be a subgroup of $K_0$ and $n \geq 3$. The action of the subgroup $K'AN \subset K_0AN = Q$ of $SO^o_{n,1}$ on $\M^{n+1}$ is of cohomogeneity one. The orbits are:
\begin{itemize}
\item[(i)] The hyperbolic spaces $H^n_+(r)$ and $H^n_-(r)$, $r \in \R_+$;
\item[(ii)] The single point $\{0\}$, the two open rays $\R_+w_0$ and $\R_-w_0$ and their complements $\CM^n_+ \setminus \R_-w_0$ and $\CM^n_- \setminus \R_+w_0$, where $w_0 = e_n - e_{n+1}$;
\item[(iii)] The two connected components of the complement of the cylinder $Z^{n-1}(r)$ in $\dS^n(r)$ and
\[
\bigcup_{L \in S^{n-2}(r)/K'} L + \R w_0,
\]
where $S^{n-2}(r)/K'$ parametrizes the orbits of the $K'$-action on $S^{n-2}(r) \subset Z^{n-1}(r)$, $r \in \R_+$.
\end{itemize}
\end{theorem}

We denote by $\FM_{K'AN}$ the decomposition of $\M^{n+1}$ into the orbits of the action of $K'AN$.

We see from Theorem \ref{nparabolic} that the orbits of $K'AN$ on $\M^{n+1} \setminus \W^n$ are independent of the choice of $K'$. Thus we get the following remarkable consequence of Theorem \ref{nparabolic}:

\begin{corollary}\label{denseopen}
There exist cohomogeneity one actions on $\M^{n+1}$, $n \geq 3$, which are orbit-equivalent on the complement of an $n$-dimensional degenerate subspace $\W^n$ of $\M^{n+1}$ and not orbit-equivalent on $\W^n$.
\end{corollary}

Thus, even if the orbit structure of a cohomogeneity one action on $\M^{n+1}$ is known on an open and dense subset of $\M^{n+1}$, it does not necessarily determine the orbit structure on the entire space $\M^{n+1}$. Such a phenomenon cannot occur in Riemannian geometry. The orbit structure of a cohomogeneity one action on a Riemannian manifold is uniquely determined by a single orbit of the action.

\section{Cohomogeneity one actions on $\M^{2}$}\label{sect:L2}

In this section we classify cohomogeneity one actions on the $2$-dimensional Minkowski space $\M^2$ up to orbit-equivalence. We start by discussing some examples of such actions. Let $H$ be a connected subgroup of $G = I^o(\M^2) = SO^o_{1,1} \ltimes \M^2$ acting on $\M^2$ with cohomogeneity one. We denote by $\g{h} \subset \g{g} = \g{so}_{1,1} + \M^2$ the Lie algebra of $H$ and by $\FM_{H}$ the (possibly singular) foliation of $\M^2$ by the orbits of the action of $H$.

Let $H \in \{\R^1,\M^1,\W^1\}$ with $\R^1 = \R e_1$, $\M^1 = \R e_2$ and $\W^1 = \R(e_1 - e_2)$. Then $H$ is a one-dimensional subgroup of the translation group $\M^2 \subset G$. The orbits of $H$ are the affine lines in $\M^2$ that are parallel to the line $\R^1$, $\M^1$ and $\W^1$ respectively and hence form a totally geodesic foliation of $\M^2$. The orbit space is isomorphic to $\R$.

Let $H = SO^o_{1,1}$. As we saw in Section \ref{sect:isotropy action}, $\FM_{SO^o_{1,1}}$ consists of the single point $\{0\}$, the four half-lines $\CM^1_{\pm\pm}$ and the hyperbolas $H^1_\pm(r)$ and $\mathbb{S}^1_\pm(r)$, $r \in \R_+$. The orbit space is a non-Hausdorff space and isomorphic to the union of five points and four copies of $\R$.

\begin{theorem}\label{th:L2}
Let $H$ be a connected subgroup of $I^o(\M^2)$ acting on $\M^2$ with cohomogeneity one. Then the action of $H$ is orbit-equivalent to the action of $\R^1$, $\M^1$, $\W^1$ or $SO^o_{1,1}$.
\end{theorem}

\begin{proof}
The intersection $\g{h}\cap\M^2$ of $\g{h}$ with the translation part of $\g{g}$ is either zero- or one-dimensional.

If $\g{h}\cap\M^2$ is one-dimensional and time-like or space-like, we must have $H=H\cap\M^2$ since $SO^o_{1,1}$ acts transitively on the set of time-like lines and space-like lines respectively. In this case the action of $H$ is orbit-equivalent to the action of $\R^1$ or $\M^1$. If $\g{h}\cap\M^2$ is one-dimensional and light-like, we can assume that $\g{h}\cap\M^2 = \W^1$ since $O_{1,1}$ acts transitively on the set of the two light-like lines in $\M^2$. In this case we obtain that $H = \W^1$ or $H = SO^o_{1,1} \ltimes \W^1$. However, the latter group has only three orbits on $\M^2$, namely the line $\W^1$ and the two open half-planes in $\M^2$ bounded by it. Thus we must have $H = \W^1$.

If $\g{h}\cap\M^2$ is zero-dimensional, then we have $\g{h} = \R(Y+v)$ with $Y = \begin{pmatrix} 0 & 1 \\ 1 & 0 \end{pmatrix} \in \g{so}_{1,1}$ and $v \in \M^2$. Thus $H$ is of the form $H = \{\Exp(t(Y+v)) : t \in \R\}$. Since $\Ad((I_2,Yv))(Y+v) = Y + (v - Y^2v) = Y$ we have $\Ad((I_2,Yv))(\g{h}) = \R Y = \g{so}_{1,1}$. This implies that the actions of $H$ and $SO^o_{1,1}$ on $\M^2$ are conjugate and hence orbit-equivalent.
\end{proof}

\section{Cohomogeneity one actions on $\M^{3}$}\label{sect:L3}

In this section we classify cohomogeneity one actions on the $3$-dimensional Minkowski space $\M^3$ up to orbit-equivalence. We first fix some notations. The identity component $G = I^o(\M^3)$ of the full isometry group of $\M^3$ is the semidirect product $G = SO^o_{2,1} \ltimes \M^3$ and its Lie algebra $\g{g}$ is the semidirect sum $\g{so}_{2,1} + \M^3$ (see Section \ref{sect:preliminaries}). As described in Section~\ref{sect:preliminaries}, we denote by $SO^o_{2,1}= KAN$ the Iwasawa composition of $SO^o_{2,1}$ and by $\g{so}_{2,1} = \g{k} + \g{a} + \g{n}$ the Iwasawa decomposition of $\g{so}_{2,1}$. We define
\begin{align*}
Y_{\g{k}} &= \begin{pmatrix} 0 & -1 & 0 \\ 1 & 0 & 0 \\ 0 & 0 & 0 \end{pmatrix} \in \g{k}\,,
&Y_{\g{a}} &= \begin{pmatrix} 0 & 0 & 0 \\ 0 & 0 & -1 \\ 0 & -1 & 0 \end{pmatrix} \in \g{a}\,,
&Y_{\g{n}} &= \begin{pmatrix} 0 & 1 & 1 \\ -1 & 0 & 0 \\ 1 & 0 & 0 \end{pmatrix} \in \g{n}\,.
\end{align*}
Then we have $[Y_{\g{k}},Y_{\g{a}}] = Y_{\g{k}} + Y_{\g{n}}$,
$[Y_{\g{k}},Y_{\g{n}}] = -Y_{\g{a}}$,
$[Y_{\g{a}},Y_{\g{n}}] = Y_{\g{n}}$,
\begin{align*}
Y_{\g{k}}u &= ( -u_2 , u_1 , 0 )^t\ ,
&Y_{\g{a}}u &= (0 , -u_3 , -u_2 )^t\ ,
&Y_{\g{n}}u &= ( u_2+u_3 , -u_1 , u_1 )^t\ ,
\end{align*}
and
\begin{align*}
k_t = \Exp(tY_{\g{k}}) & = \begin{pmatrix} \cos(t) & -\sin(t) & 0 \\ \sin(t) & \cos(t) & 0 \\ 0 & 0 & 1 \end{pmatrix} \in K \ ,\\
a_t = \Exp(tY_{\g{a}}) & = \begin{pmatrix} 1 & 0 & 0 \\ 0 & \cosh(t) & -\sinh(t) \\ 0 & -\sinh(t) & \cosh(t) \end{pmatrix}\in A\ , \\
n_t = \Exp(tY_{\g{n}}) & = \begin{pmatrix} 1 & t & t \\ -t & 1-\frac{1}{2}t^2 & -\frac{1}{2}t^2  \\ t & \frac{1}{2}t^2 & 1 + \frac{1}{2}t^2 \end{pmatrix}\in N\ . \\
\end{align*}
Finally, we define the light-like line $\ell \subset \M^3$ by $\ell = \R(e_2-e_3)$.

We start by discussing some examples of cohomogeneity one actions on $\M^3$. For $H \subset G$ we denote by $\FM_{H}$ the collection of orbits of the action of $H$ on $\M^3$. We will distinguish several types of actions.

\medskip
Type (I): \textit{$\FM_{H}$ is invariant under a two-dimensional translation group.}

Let $\g{h} \in \{\R^2,\M^2,\W^2\}$. Then $\g{h}$ is a two-dimensional abelian subalgebra of the translation algebra $\M^3$ in $\g{g}$. The corresponding connected subgroup $H$ of $G$ acts on $\M^3$ with cohomogeneity one and $\FM_H$ is a totally geodesic foliation of $\M^3$ whose leaves consist of the affine planes in $\M^3$ that are parallel to $\R^2$, $\M^2$ and $\W^2$ respectively:
\begin{align*}
\FM_{\R^2} &= \bigcup_{t \in \R} (te_3 + \R^2)\ ,
& \FM_{\M^2} &= \bigcup_{t \in \R} (te_1 + \M^2)\ ,
& \FM_{\W^2} &= \bigcup_{t \in \R} (t(e_2+e_3) + \W^2).
\end{align*}
The orbit space is isomorphic to $\R$.

\medskip
Type (II): \textit{$\FM_{H}$ is invariant under a one-dimensional translation group.}

\medskip
Type (II)$_s$: \textit{$\FM_{H}$ is invariant under a one-dimensional space-like translation group.}

The set $\g{h} = \g{a} \oplus \R e_1$
is a two-dimensional abelian subalgebra of $\g{g}$ and $H = \Exp(\g{h}) = A \times \R e_1$ is a two-dimensional connected abelian subgroup of $G$. The action of $A$ leaves the foliation $\FM_{\M^2}$ invariant and on each leaf $te_1 + \M^2 \in \FM_{\M^2}$ the orbits consist of the single point $\{te_1\}$, the four half-lines $te_1 + \CM^1_{\pm\pm}$, the two hyperbolas $te_1 + H^1_\pm(r)$ and the two hyperbolas $te_1 + \dS^1_\pm(r)$, $r \in \R_+$. Thus the orbits of $H$ are the line $\R e_1$, the four half-planes $\R e_1 \times \CM^1_{\pm\pm}$ and the hyperbolic cylinders $\R e_1 \times H^1_\pm(r)$ and $\R e_1 \times \dS^1_\pm(r)$, $r \in \R_+$:
\[
\FM_{A \times \R e_1} = \R e_1 \cup (\R e_1 \times \CM^1_{\pm\pm}) \cup \bigcup_{r \in \R_+} (\R e_1 \times H^1_\pm(r)) \cup \bigcup_{r \in \R_+} (\R e_1 \times \dS^1_\pm(r)).
\]
 The orbit space is a non-Hausdorff space and isomorphic to the union of five points and four copies of $\R_+$.

\medskip
Type (II)$_t$: \textit{$\FM_{H}$ is invariant under a one-dimensional time-like translation group.}

The set $\g{h} = \g{k} \oplus \R e_3$
is a two-dimensional abelian subalgebra of $\g{g}$ and $H = \Exp(\g{h}) = K \times \R e_3$ is a two-dimensional connected abelian subgroup of $G$. The action of $K$ leaves the foliation $\FM_{\R^2}$ invariant. On each leaf $te_3 + \R^2 \in \FM_{\R^2}$ the orbits consist of the single point $\{te_3\}$ and the circles centered at that point. The orbits of $H$ therefore consist of the time-like subspace $\R e_3$ and the cylinders $S^1(r) \times \R e_3 \subset \R^2 \times \R e_3$, where $S^1(r)$ is the circle of radius $r \in \R_+$ in $\R^2$:
\[
\FM_{K \times \R e_3} = \R e_3 \cup \bigcup_{r \in \R_+} (S^1(r) \times \R e_3) .
\]
The orbit space is isomorphic to the closed interval $[0,\infty)$.

\medskip
Type (II)$_l$: \textit{$\FM_{H}$ is invariant under a one-dimensional light-like translation group.}

For $\lambda \geq 0$ we define $\g{a}_\lambda = \R(Y_{\g{a}} + \lambda e_1)$ and $\g{n}_\lambda = \R(Y_{\g{n}} + \lambda e_3)$,
and denote by $A_\lambda = \Exp(\g{a}_\lambda)$ and  $N_\lambda = \Exp(\g{n}_\lambda)$
the corresponding one-dimensional subgroups of $G$. For $\lambda = 0$ we have $\g{a}_0 = \g{a}$, $A_0 = A$, $\g{n}_0 = \g{n}$ and $N_0 = N$. Explicitly, we have
\[
\Exp(t(Y_{\g{a}} + \lambda e_1)) =
\left( \begin{pmatrix} 1 & 0 & 0 \\ 0 & \cosh(t) & -\sinh(t) \\ 0 & -\sinh(t) & \cosh(t) \end{pmatrix} ,
\begin{pmatrix} \lambda t \\ 0 \\ 0 \end{pmatrix} \right) \in SO^o_{2,1} \ltimes \M^3
\]
and
\[
\Exp(t(Y_{\g{n}} + \lambda e_3)) =
\left( \begin{pmatrix} 1 & t & t \\ -t & 1-\frac{1}{2}t^2 & -\frac{1}{2}t^2  \\ t & \frac{1}{2}t^2 & 1 + \frac{1}{2}t^2 \end{pmatrix} ,
\begin{pmatrix} \frac{1}{2}\lambda t^2 \\ -\frac{1}{6}\lambda t^3 \\ \lambda t + \frac{1}{6}\lambda t^3 \end{pmatrix} \right) \in SO^o_{2,1} \ltimes \M^3.
\]

We have $[\g{a}_\lambda,\ell] = \ell$ and $[\g{n}_\lambda,\ell] = 0$. Therefore $\g{a}_{\lambda,\ell} = \g{a}_\lambda \oplus \ell$ and $\g{n}_{\lambda,\ell} = \g{n}_\lambda \oplus \ell$
are subalgebras of $\g{g}$, $\g{a}_{\lambda,\ell}$ is the semidirect sum of $\g{a}_\lambda$ and  $\ell$ and a solvable subalgebra of $\g{g}$, and $\g{n}_{\lambda,\ell}$ is the direct sum of $\g{n}_\lambda$ and  $\ell$ and an abelian subalgebra of $\g{g}$.
We define $A_{\lambda,\ell} = \Exp(\g{a}_{\lambda,\ell})$ and $N_{\lambda,\ell} = \Exp(\g{n}_{\lambda,\ell})$,
which are the connected subgroups of $G$ with Lie algebra $\g{a}_{\lambda,\ell}$ and $\g{n}_{\lambda,\ell}$, respectively. The solvable Lie group $A_{\lambda,\ell}$ is the semidirect product $A_{\lambda,\ell} = A_\lambda \ltimes \ell$ and explicitly given by
\[
A_\lambda \ltimes \ell = \left\{ g_{t,s}^\lambda = \left(
\begin{pmatrix} 1 & 0 & 0 \\ 0 & \cosh(t) & -\sinh(t) \\ 0 & -\sinh(t) & \cosh(t) \end{pmatrix} ,
\begin{pmatrix} \lambda t \\ s \\ -s \end{pmatrix} \right)  : t,s \in \R \right\} \subset SO^o_{2,1} \ltimes \M^3.
\]
The abelian Lie group $N_{\lambda,\ell}$ is the direct product $N_{\lambda,\ell} = N_\lambda \times \ell$ and explicitly given by
\[
N_\lambda \times \ell = \left\{ h_{t,s}^\lambda = \left(
\begin{pmatrix} 1 & t & t \\ -t & 1-\frac{1}{2}t^2 & -\frac{1}{2}t^2  \\ t & \frac{1}{2}t^2 & 1 + \frac{1}{2}t^2 \end{pmatrix} ,
\begin{pmatrix} \frac{1}{2}\lambda t^2 \\ s-\frac{1}{6}\lambda t^3 \\ \lambda t + \frac{1}{6}\lambda t^3 - s \end{pmatrix} \right)  : t,s \in \R \right\} \subset SO^o_{2,1} \ltimes \M^3.
\]
We will now discuss the orbit structure of these actions in more detail.

We first consider the case $\lambda = 0$.
We have
\[
g_{t,s}^0(xe_1 + y(e_2-e_3) + z(e_2+e_3)) = xe_1 + (s+e^ty)(e_2-e_3) + e^{-t}z(e_2+e_3),
\]
which shows that $A \ltimes \ell$ leaves the foliation $\FM_{\M^2}$ invariant. On each leaf $xe_1 + \M^2$ there are precisely three orbits, namely the line $xe_1 + \ell$ and the two open half-planes $xe_1 + \{u \in \M^2:u_2+u_3 > 0\}$ and $xe_1 + \{u \in \M^2:u_2+u_3 < 0\}$ bounded by that line. Thus the set of orbits is
\[
\FM_{A \ltimes \ell} = \bigcup_{x \in \R} \left((xe_1 + \ell) \cup  (xe_1 + \{u \in \M^2:u_2+u_3 > 0\}) \cup  (xe_1 + \{u \in \M^2:u_2+u_3 < 0\})\right).
\]
Hence the orbit space is isomorphic to the union of three copies of $\R$.

Next, we have
\[
h_{t,s}^0(xe_1 + y(e_2-e_3) + z(e_2+e_3)) = (x+2tz)e_1 + (s+y - tx -t^2z)(e_2-e_3) + z(e_2+e_3),
\]
which shows that $N \times \ell$ leaves the foliation $\FM_{\W^2}$ invariant. On the leaf $\W^2$ the orbits are the lines $xe_1 + \ell$, and on the leaf $z(e_2+e_3) + \W^2$, $z \neq 0$, the action is transitive. Thus the set of orbits is
\[
\FM_{N \times \ell} = \bigcup_{x \in \R} (xe_1 + \ell) \cup (\FM_{\W^2} \setminus \W^2).
\]
Hence the orbit space is isomorphic to the union of three copies of $\R$.

We now assume that $\lambda > 0$.
We have $g_{t,s}^\lambda(0) = \lambda t e_1 + s(e_2-e_3)$
and therefore $(A_\lambda \ltimes \ell) \cdot 0 = \W^2$. More general, we have
\[
g_{t,s}^\lambda(xe_1 + y(e_2-e_3) + z(e_2+e_3)) = (x+\lambda t) e_1 + (s + ye^t)(e_2-e_3) + ze^{-t}(e_2+e_3) .
\]
Given $p = xe_1 + y(e_2-e_3) + z(e_2+e_3) \in \M^3$, we put $t = -x/\lambda$ and $s = -ye^{-x/\lambda}$. Then $g_{t,s}^\lambda(p) = ze^{x/\lambda}(e_2+e_3)$. This shows that $\R(e_2+e_3)$ intersects each orbit of the action of $A_\lambda \ltimes \ell$. Conversely, we have
\[
g_{t,s}^\lambda(z(e_2+e_3)) = \lambda t e_1 + s (e_2-e_3) + ze^{-t}(e_2+e_3) \in \R(e_2+e_3) \Longleftrightarrow (t,s) = (0,0),
\]
which implies that $\R(e_2+e_3)$ intersects each orbit of the action of $A_\lambda \ltimes \ell$ exactly once. Thus $\R(e_2+e_3)$ parametrizes the orbit space of the action of $A_\lambda \ltimes \ell$ on $\M^3$. We also have
\begin{align*}
& a_u(g_{t,s}^\lambda(xe_1 + y(e_2-e_3) + z(e_2+e_3))) \\ & =  (x+\lambda t) e_1 + e^u(s + ye^t)(e_2-e_3) + e^{-u}ze^{-t}(e_2+e_3) \\
& = g_{t,e^us}^\lambda(xe_1 + e^uy(e_2-e_3) + e^{-u}z(e_2+e_3)) \\
& =   g_{t,e^us}^\lambda(a_u(xe_1 + y(e_2-e_3) + z(e_2+e_3)),
\end{align*}
that is,
$a_u \circ g_{t,s}^\lambda = g_{t,e^us}^\lambda \circ a_u$.
It follows that all isometries in the abelian Lie group $A$ map orbits of $A_\lambda \ltimes \ell$ onto orbits of $A_\lambda \ltimes \ell$. Since $a_u(e_2+e_3) = e^{-u}(e_2+e_3)$ we see that the orbits parametrized by $\R_+(e_2+e_3)$ are isometrically congruent to each other under the action of $A$, and the same is true for the orbits parametrized by $\R_-(e_2+e_3)$. The isometry of $\M^3$ given by $e_1 \mapsto e_1$, $e_2 \mapsto -e_2$ and $e_3 \mapsto -e_3$ maps the orbit through $z(e_2+e_3)$ onto the orbit through $-z(e_2+e_3)$. It follows that all orbits different from $\W^2 = (A_\lambda \ltimes \ell) \cdot 0$ are isometrically congruent to each other.

We now investigate orbit-equivalence of the actions of $A_\lambda \ltimes \ell$, $\lambda >0$. Let $\lambda,\mu > 0$, $\lambda \neq \mu$, and assume that the actions of $A_\lambda \ltimes \ell$ and $A_\mu \ltimes \ell$ are orbit-equivalent. Both actions have exactly one degenerate orbit, namely
$(A_\lambda \ltimes \ell) \cdot 0 = \W^2 = (A_\mu \ltimes \ell) \cdot 0$. Thus any isometry
$g$ of $\M^3$ mapping the orbits of $A_\lambda \ltimes \ell$ onto the orbits of $A_\mu \ltimes \ell$ must necessarily satisfy $g(\W^2) = \W^2$. The subgroup of $I(\M^3)$ leaving $\W^2$ invariant is $AN \ltimes \W^2 \cup -I_3(AN \ltimes \W^2)$. We first assume that $g \in AN \ltimes \W^2$. Since $A$ normalizes $N$ and $A$ maps orbits of $A_\lambda \ltimes \ell$  onto orbits of $A_\lambda \ltimes \ell$, we can assume that $g \in N \ltimes \W^2$. Then $g$ maps the orbit $(A_\lambda \ltimes \ell) \cdot (e_2+e_3)$ onto one of the orbits of $A_\mu \ltimes \ell$, that is,
$g((A_\lambda \ltimes \ell) \cdot (e_2+e_3)) = (A_\mu \ltimes \ell) \cdot z(e_2+e_3)$
for some $0 \neq z \in \R$. If we write $g = (n_\beta,w) \in N \ltimes \W^2$, the previous equation can be written as
$n_\beta(g^\lambda_{t,s}(e_2+e_3)) + w = g^\mu_{v,u}(z(e_2+e_3))$,
where $\beta \in \R$, $v = v(t,s)$, $u = u(t,s)$ and $w = w_1e_1 + w_2(e_2-e_3) \in \W^2$. Comparing the coefficients corresponding to $e_1,e_2-e_3,e_2+e_3$, respectively, leads to the three equations
\begin{align*}
\lambda t + 2\beta e^{-t} + w_1  &= \mu v\ ,
&s - \beta \lambda t - \beta^2 e^{-t} + w_2  &= u\ ,
&e^{-t}  &= ze^{-v}.
\end{align*}
From the first equation we see that $v = v(t)$ is a function of $t$ only. The third equation
implies $z = e^{v-t}$. Since $z$ is a constant, the third equation implies $v^\prime = 1$, where differentiation is with respect to $t$. From the first equation we get $v^\prime = \mu^{-1}(\lambda - 2\beta e^{-t})$. Comparing both equations for $v^\prime$ leads to $\beta = 0$ and $\lambda = \mu$, which contradicts the assumption $\lambda \neq \mu$. Thus $N \ltimes \W^2$ maps orbits of $A_\lambda \ltimes \ell$  onto orbits of $A_\lambda \ltimes \ell$. It is easy to see that $-I_3$ maps
orbits of $A_\lambda \ltimes \ell$  onto orbits of $A_\lambda \ltimes \ell$ as well. Altogether we can now conclude that the actions of $A_\lambda \ltimes \ell$ and $A_\mu \ltimes \ell$ for $\lambda,\mu > 0$ are orbit-equivalent if and only if $\lambda = \mu$.

Now we study the orbits of $N_{\lambda,\ell}$. We denote by $P_\lambda = \{h_{t,s}^\lambda(0) : t,s \in \R\}$ the orbit $(N_\lambda \times \ell) \cdot 0$. We have
\[
h_{t,s}^\lambda(0) =  \frac{1}{2}\lambda t^2 e_1 + \left(  s -\frac{1}{6}\lambda t^3 \right) (e_2-e_3) + \lambda t e_3
\]
and
\[
h_{t,s}^\lambda(xe_1) = xn_t(e_1) + h_{t,s}^\lambda(0) = xe_1 - xt(e_2-e_3) + h_{t,s}^\lambda(0) = xe_1 + h_{t,s-xt}^\lambda(0),
\]
and therefore $(N_\lambda \times \ell) \cdot xe_1 = xe_1 + P_\lambda$ for all $x \in \R$. The first of these equations shows that $(N_\lambda \times \ell) \cdot 0$ is the ruled surface $P_\lambda$ in $\M^3$ generated by the parabola $z \mapsto \frac{1}{2\lambda}z^2 e_1 + ze_3$ and ruled by the light-like lines $\ell$.  It also follows that the set of orbits is
\[
\FM_{N_\lambda \times \ell} = \bigcup_{x \in \R} (xe_1 + P_\lambda).
\]
Thus all orbits of $N_\lambda \times \ell$ are isometrically congruent to each other.
The orbit space is isomorphic to $\R$. Clearly, this implies that the action of $N_{0,\ell}$ cannot be orbit-equivalent to the action of $N_{\lambda,\ell}$ with $\lambda > 0$. We can rewrite the above expression for $h_{t,s}^\lambda(0)$ as
\[
h_{t,s}^\lambda(0) =  \frac{1}{2}\lambda t^2 e_1 + \left(  s -\frac{1}{2}\lambda t - \frac{1}{6}\lambda t^3 \right) (e_2-e_3) + \frac{1}{2}\lambda t (e_2+e_3).
\]
Now let $\lambda,\mu > 0$ and put $u = \frac{1}{2}\ln\left(\frac{\lambda}{\mu}\right)$, $t' = e^ut$ and $s' = e^us - \sinh(u)\lambda t$. Then we get
\begin{align*}
a_u(h_{t,s}^\lambda(0)) & =  \frac{1}{2}\lambda t^2 e_1 + \left(  s -\frac{1}{2}\lambda t - \frac{1}{6}\lambda t^3 \right) e^u(e_2-e_3) + \frac{1}{2}\lambda t e^{-u}(e_2+e_3) \\
& =  \frac{1}{2}\mu t'^2 e_1 + \left(  s' -\frac{1}{2}\mu t' - \frac{1}{6}\mu t'^3 \right)(e_2-e_3) + \frac{1}{2}\mu t' (e_2+e_3) =  h_{t's'}^\mu(0),
\end{align*}
which shows that $a_u(P_\lambda) = P_\mu$. It follows that for all $\lambda,\mu > 0$ the actions of
$N_\lambda \times \ell$ and of $N_\mu \times \ell$ are orbit-equivalent.

\medskip
Type (III) \textit{$\FM_{H}$ is not invariant under any translation group.}

The restricted Lorentz group $SO^o_{2,1}$ and its parabolic subgroup $AN$ act with cohomogeneity one on $\M^3$. We discussed these two actions in detail in Sections \ref{sect:isotropy action} and
\ref{sect:parabolic action}. It follows in particular that for each of these two groups the set of orbits is not invariant under any translation group.

\smallskip
We can  now formulate the main classification result.

\begin{theorem}\label{th:L3}
Let $H$ be a connected subgroup of $I^o(\M^3)$ acting on $\M^3$ with cohomogeneity one. Then the action of $H$ is orbit-equivalent to one of the following actions:
\begin{enumerate}[{\rm (1)}]
\item $\R^2$, $\M^2$ or $\W^2$;
\item $K \times \R e_3$, $A \times \R e_1$, $N \times \ell$, $N_1 \times \ell$ or $A_\lambda \ltimes \ell$ ($\lambda \geq 0$);
\item $SO^o_{2,1} = KAN$ or $AN$;
\end{enumerate}
\end{theorem}

\begin{proof}
We first consider the case $\dim(\g{h} \cap \M^3) = 2$. Every two-dimensional Riemannian, Lorentz\-ian and degenerate subspace of $\M^3$ is conjugate under $O_{2,1}$ to $\R^2$, $\M^2$ and $\W^2$, respectively. For dimension reasons it follows that the action of $H$ is orbit-equivalent to the action of one of the three translation subgroups $\R^2$, $\M^2$ or $\W^2$.

Next, we consider the case $\dim(\g{h} \cap \M^3) = 0$. The projection $\pi_1\colon\g{so}_{2,1} \oplus_\phi\M^3\to\g{so}_{2,1}$ is a Lie algebra homomorphism and therefore $\pi_1(\g{h})$ is a subalgebra of $\g{so}_{2,1}$.  Since $\dim(\g{h} \cap \M^3) = 0$, we must have $\dim(\pi_1(\g{h})) \geq 2$. Every subalgebra of $\g{so}_{2,1}$ of dimension $\geq 2$ is conjugate to $\g{so}_{2,1}$ (for dimension $3$) or $\g{a} \oplus \g{n}$ (for dimension $2$).

If $\pi_1(\g{h}) = \g{so}_{2,1}$, there exist $u,v,w \in \M^3$ such that
$\g{h} = \R(Y_{\g{k}} + u)  + \R(Y_{\g{a}} + v) + \R(Y_{\g{n}} + w)$.
Since $\g{h}$ is a subalgebra, we get
\begin{align*}
\lbrack Y_{\g{k}} + u,Y_{\g{a}} + v \rbrack & = Y_{\g{k}} + Y_{\g{n}} + (Y_{\g{k}}v - Y_{\g{a}}u) \in \g{h},\\
\lbrack Y_{\g{k}} + u,Y_{\g{n}} + w \rbrack & = -Y_{\g{a}} + (Y_{\g{k}}w - Y_{\g{n}}u) \in \g{h},\\
\lbrack Y_{\g{a}} + v,Y_{\g{n}} + w \rbrack & = Y_{\g{n}} + (Y_{\g{a}}w - Y_{\g{n}}v) \in \g{h}.
\end{align*}
This implies $Y_{\g{k}}v - Y_{\g{a}}u = u+w$, $Y_{\g{k}}w - Y_{\g{n}}u = -v$, $Y_{\g{a}}w - Y_{\g{n}}v = w$,
or equivalently,
\begin{align*}
\begin{pmatrix} -v_2 \\ u_3+v_1 \\  u_2 \end{pmatrix}
&= \begin{pmatrix}  u_1+w_1\\ u_2+w_2 \\ u_3+w_3 \end{pmatrix},
&\begin{pmatrix} u_2 + u_3 + w_2 \\ -u_1 - w_1 \\ u_1 \end{pmatrix}
&= \begin{pmatrix} v_1 \\ v_2 \\ v_3 \end{pmatrix},
&\begin{pmatrix} -v_2 -  v_3 \\ v_1 - w_3 \\- v_1 - w_2 \end{pmatrix}
&= \begin{pmatrix} w_1 \\ w_2 \\ w_3 \end{pmatrix}.
\end{align*}
These equations lead to
$u = (u_1 , u_2 , 0)^t$,
$v = ( 0 , v_2 , u_1 )^t$,
$w = (-u_1 - v_2 , -u_2 , u_2 )^t$.
Since
\begin{align*}
\Ad((I_3,(u_2,-u_1,-v_2)^t))(Y_{\g{k}} + u) & = Y_{\g{k}},\\
\Ad((I_3,(u_2,-u_1,-v_2)^t))(Y_{\g{a}} + v) & = Y_{\g{a}},\\
\Ad((I_3,(u_2,-u_1,-v_2)^t))(Y_{\g{n}} + w) & = Y_{\g{n}},
\end{align*}
we get $\Ad((I_3,(-u_2,u_1,v_2)^t))(\g{g}) = \g{so}_{2,1}$.
This shows that the action of $H$ is orbit-equivalent to the action of $SO^o_{2,1}$.

If $\pi_1(\g{h}) = \g{a} \oplus \g{n}$, there exist $v,w \in \M^3$ such that
$\g{h} = \R(Y_{\g{a}} + v) + \R(Y_{\g{n}} + w)$.
Since $\g{h}$ is a subalgebra, we get
$\lbrack Y_{\g{a}} + v,Y_{\g{n}} + w \rbrack  =  Y_{\g{n}} + (Y_{\g{a}}w - Y_{\g{n}}v) \in \g{h}$.
This implies $Y_{\g{a}}w - Y_{\g{n}}v = w$, or equivalently,
$-v_2 - v_3 = w_1$, $v_1 - w_3 = w_2$, $-v_1  - w_2 = w_3$.
These equations lead to
$v = ( 0 , v_2 , v_3 )^t$ and $w = (-v_2 - v_3 , w_2 , - w_2 )^t$.
Since
$\Ad((I_3,(-w_2,-v_3,-v_2)^t))(Y_{\g{a}} + v) = Y_{\g{a}}$ and
$\Ad((I_3,(-w_2,-v_3,-v_2)^t))(Y_{\g{n}} + w)  = Y_{\g{n}}$,
we get
$\Ad((I_3,(-w_2,-v_3,-v_2)^t))(\g{h}) = \g{a} \oplus \g{n}$.
This shows that the action of $H$ is orbit-equivalent to the action of $AN$.

We now consider the case $\dim(\g{h} \cap \M^3) = 1$. If $\dim(\pi_1(\g{h})) \in \{2,3\}$, then we have $\pi_1(\g{h}) = \g{so}_{2,1}$ or $\pi_1(\g{h}) = \g{a} \oplus \g{n}$. Using the same arguments as for the case $\dim(\g{h} \cap \M^3) = 0$ we see that $\g{h}$ is conjugate to $\g{so}_{2,1} \oplus \R u$ or $(\g{a} \oplus \g{n}) \oplus\R u$ with some $0 \neq u \in \M^3$ respectively. The hyperbolic plane $H^2_+(1) = SO^o_{2,1} \cdot e_3= AN \cdot e_3$ does not contain a line and therefore the orbit $(SO^o_{2,1} \ltimes \R u) \cdot e_3 =  (AN \ltimes \R u) \cdot e_3$ is three-dimensional. Thus we must have $\dim(\pi_1(\g{h})) \leq 1$. If $\dim(\pi_1(\g{h})) = 0$, then $\g{h} \subset \M^3$ and therefore $H$ is a one-dimensional translation group. Such a group acts with cohomogeneity two, which gives a contradiction. We conclude that  $\dim(\pi_1(\g{h})) = 1$.

Every one-dimensional space-like, time-like and light-like subspace of $\M^3$ is conjugate under $O_{2,1}$ to $\R e_1$, $\R e_3$ and $\ell$, respectively. We can therefore assume that $\g{h} \cap \M^3$ is equal to one of these three one-dimensional subspaces.

Assume that $\g{h} \cap \M^3 = \R e_1$. The normalizer of $\R e_1$ in $\g{so}_{2,1} = \g{k} \oplus \g{a} \oplus \g{n}$ is equal to~$\g{a}$, which implies $\pi_1(\g{h}) = \g{a}$. Thus we have
$\g{h} = \R(Y_{\g{a}}+v) \oplus \R e_1$ with $v  \in \M^3$.
Since $\Ad((I_3,(0,-v_3,-v_2)^t))(\g{h}) = \g{a} \oplus \R e_1$, we conclude that the action of $H$ is orbit-equivalent to the action of $A \times \R e_1$.

Assume that $\g{h} \cap \M^3 = \R e_3$. The normalizer of $\R e_3$ in $\g{so}_{2,1} = \g{k} \oplus \g{a} \oplus \g{n}$ is equal to~$\g{k}$, which implies $\pi_1(\g{h}) = \g{k}$. Thus we have
$\g{h} = \R(Y_{\g{k}}+v) \oplus \R e_3$ with $v \in \M^3$.
Since $\Ad((I_3,(v_2,-v_1,0)^t))(\g{h}) = \g{k} \oplus \R e_3$, we conclude that the action of $H$ is orbit-equivalent to the action of $K \times \R e_3$.

Assume that $\g{h} \cap \M^3 = \ell$. The normalizer of $\ell$ in $\g{so}_{2,1} = \g{k} \oplus \g{a} \oplus \g{n}$ is equal to $\g{a} \oplus \g{n}$, which implies $\pi_1(\g{h}) \subset \g{a} \oplus \g{n}$. Thus we have $\g{h} = \R(aY_{\g{a}} + bY_{\g{n}}+u) \oplus \ell$ with
$a,b \in \R$ and  $u  \in \M^3$,
where $a,b$ are not both equal to $0$.

We first consider the case $a \neq 0$. Then $\Ad((n_{-b/a},0))(\g{h})$ is  of the form
$\R(Y_{\g{a}} +v) \oplus \ell$ with $v  \in \M^3$.
Since $\Ad((I_3,(0,-v_3,-v_2)^t))(Y_{\g{a}}+v) =  Y_{\g{a}} + v_1e_1$ and $\Ad((I_3,u))(e_2-e_3) =  e_2-e_3$,
we see that $\g{h}$ is conjugate to a subalgebra of the form $\R(Y_{\g{a}}+\lambda e_1) \oplus \ell$ with $\lambda \in \R$.
It follows that the action of $H$ is orbit-equivalent to the action of $A_\lambda \ltimes \ell$ for $\lambda \geq 0$. (For $\lambda < 0$ use the transformation $e_1 \mapsto -e_1$, $e_2 \mapsto e_2$  and $e_3 \mapsto e_3$.)

Next, we consider the case $a = 0$. Then we can assume that $b = 1$ and hence that $\g{h}$ is of the form $\g{h} = \R(Y_{\g{n}} +v) \oplus \ell$ with $v  \in \M^3$.
For $u = (-v_2,v_1,0)^t$ we get
$\Ad((I_3,u))(Y_{\g{n}}+v) =  Y_{\g{n}} + (v_2+v_3)e_3$ and $\Ad((I_3,u))(e_2-e_3) =  e_2-e_3$.
It follows that $\g{h}$ is conjugate to a subalgebra of the form
$\R(Y_{\g{n}}+\lambda e_3) \oplus \ell $ with $ \lambda \in \R$
and therefore the action of $H$ is orbit-equivalent to the action of $N_\lambda \times \ell$ for $\lambda \geq 0$. (For $\lambda < 0$ use the transformation $e_1 \mapsto e_1$, $e_2 \mapsto e_2$  and $e_3 \mapsto -e_3$.)
\end{proof}

\bibliographystyle{amsplain}

\end{document}